\documentclass[a4paper,12pt]{article}
\usepackage[utf8]{inputenc}
\usepackage{amssymb}
\usepackage{amsmath}
\usepackage{amsthm}
\usepackage{tikz}
\usepackage{url}

\theoremstyle{plain}
\newtheorem{defi}{Definition}[section]
\newtheorem{thm}{Theorem}[section]
\newtheorem{prop}{Proposition}[section]
\newtheorem{lem}{Lemma}[section]

\theoremstyle{remark}

\DeclareMathOperator{\Gal}{Gal}
\DeclareMathOperator{\Aut}{Aut}
\DeclareMathOperator{\Jac}{Jac}

\DeclareMathOperator{\sw}{sw}
\DeclareMathOperator{\Sw}{Sw}

\DeclareMathOperator{\Hom}{Hom}
\DeclareMathOperator{\codim}{codim}
\DeclareMathOperator{\Id}{Id}

\DeclareMathOperator{\Cen}{Z}
\DeclareMathOperator{\Ker}{Ker}
\DeclareMathOperator{\md}{mod \;}

\newcommand{\N}{\mathbb{N}}
\newcommand{\M}{M}
\newcommand{\Z}{\mathbb{Z}}

\newcommand{\Q}{\mathbb{Q}}
\newcommand{\F}{\mathbb{F}}
\newcommand{\Qpur}{\mathbb{Q}_p^{\rm ur}}
\newcommand{\Proj}{\mathbb{P}^1}
\renewcommand{\L}{L}

\newcounter{mysub}

\begin{document}
\title{Lifting Artin-Schreier covers with maximal wild monodromy}
\author{P. Chrétien}

\maketitle

\begin{abstract}
Let $k$ be an algebraically closed field  of characteristic $p > 0$. We consider the  problem of lifting $p$-cyclic covers of $\Proj_k$ as $p$-cyclic covers of the projective line over some DVR  under the condition that the wild monodromy is maximal. We answer positively the question for covers birational to $w^ p-w=tR(t)$ for some additive polynomial $R(t)$.
\end{abstract}

\section{Introduction}
\label{sec:0}
Let $(R,v)$ be a complete discrete valuation ring of mixed characteristic $(0,p)$ with fraction field $K$ containing a primitve $p$-th root of unity $\zeta_p$ and algebraically closed residue field $k$. The stable reduction theorem states that given a smooth, projective, geometrically connected curve $C/K$ of genus $g(C) \geq 2$, there exists a unique minimal Galois extension $\M/K$ called \textsl{the monodromy extension of $C/K$} such that $C_{\M}:=C \times \M$ has stable reduction over $\M$. The group $G =  \Gal(\M/K)$ is the \textsl{monodromy group of $C/K$}.

\indent Let us consider the case where  $\phi : C \to \Proj_K$ is a $p$-cyclic cover. Let $\mathcal{C}$ be the stable model of $C_{\M}/\M$ and $\Aut_k(\mathcal{C}_k)^{\#}$ be the subgroup of $\Aut_k(\mathcal{C}_k)$ of elements acting trivially on the reduction in $\mathcal{C}_k$ of the ramification locus of $\phi \times \Id_M : C_{\M}\to \Proj_{\M}$ (see \cite{Liu} 10.1.3 for the definition of the reduction map of $C_{\M}$). One derives from the stable reduction theorem the following injection 
\begin{equation}\label{injintro}
\Gal(\M/K) \hookrightarrow \Aut_k(\mathcal{C}_k)^{\#}.
\end{equation}
When the $p$-Sylow subgroups of these groups are isomorphic, one says that the \textsl{wild monodromy is maximal}. We are interested in realization of smooth covers as above such that the $p$-adic valuation of $|\Aut_k(\mathcal{C}_k)^{\#}|$ is large compared to the genus of $\mathcal{C}_k$ and having maximal wild monodromy. Moreover, we will study the ramification filtration and the Swan conductor of their monodromy extension.

Recall that a big action is a pair $(X,G)$ where $X/k$ is a smooth, projective, geometrically connected curve of genus $g(X) \geq 2$ and $G$ is a finite $p$-group of $k$-automorphisms of $X/k$ such that $|G|>\frac{2p}{p-1}g(X)$. According to \cite{LM} Theorem 1.1 II f), if $(X,G)$ is a big action, then one has that $|G| \leq \frac{4p}{(p-1)^ 2}g(X)^2$ with equality if and only if $X/k$ is birationally given by $w^p-w=tR(t)$ where $R(t) \in k[t]$ is an additive polynomial. In this case, $G$ is an extra-special $p$-group and equals the $p$-Sylow subgroup $G_{\infty,1}(X)$ of the subgroup of $\Aut_k(X)$ leaving $t=\infty$ fixed.

This motivates the following question, with the above notations, given a big action $(C,G)$ such that $|G| = \frac{4p}{(p-1)^ 2}g(X)^2$,  is it possible to find a field $K$ and a a $p$-cyclic cover $C/K$ of $\Proj_K$   such that $\mathcal{C}_k \simeq X$, that $G \simeq \Aut(\mathcal{C}_k)_1^{\#}$ is a $p$-Sylow subgroup of $\Aut(\mathcal{C}_k)^{\#}$ and the curve $C/K$ has maximal wild monodromy ?\\

Let $n \in \N^{\times}$, $q=p^n$, $\lambda= \zeta_p-1$ and ${K=\Qpur (\lambda^{1/(1+q)})}$. For any additive polynomial $R(t) \in k[t]$ of degree $q$, let $X/k$ be curve defined by ${w^p-w=tR(t)}$. In section \ref{sec:2}, we prove the following

\begin{thm}\label{prop1intro}
There exists a $p$-cyclic cover $C/K$ of $\Proj_K$ such that ${\mathcal{C}_k \simeq X}$, one has $G_{\infty,1}(X) \simeq \Aut(\mathcal{C}_k)_1^{\#}$ and the curve $C/K$ has maximal wild monodromy $M/K$. The extension $M/K$ is the decomposition field of an explicitly given polynomial and the group ${\Gal(M/K) \simeq \Aut_k(\mathcal{C}_k)^{\#}_1}$ is an extra-special $p$-group of order $pq^2$.
\end{thm}

The group $G_{\infty,1}(\mathcal{C}_k)=\Aut_k(\mathcal{C}_k)^{\#}_1$ is endowed with the ramification filtration $(G_{\infty,i}(\mathcal{C}_k))_{i\geq 0}$  which is easily seen to be :
\[ G_{\infty,0}(\mathcal{C}_k) = G_{\infty,1}(\mathcal{C}_k)\supsetneq {\rm Z}(G_{\infty,0}(\mathcal{C}_k))= G_{\infty,2}(\mathcal{C}_k) = \dots = G_{\infty,1+q}(\mathcal{C}_k) \supsetneq \lbrace 1 \rbrace. \]
Moreover, $G:=\Gal(\M/K)$ being the Galois group of a finite extension of $K$, it is endowed with the ramification filtration $(G_i)_{i \geq 0}$. Since ${G \simeq G_{\infty,1}(\mathcal{C}_k)}$ it is natural to ask for the behaviour of $(G_i)_{i \geq 0}$ under \eqref{injintro}, that is to compare $(G_i)_{i \geq 0}$ and  $(G_{\infty,i}(\mathcal{C}_k))_{i\geq 0}$. In the general case, the arithmetic is quite tedious due to the expression of the lifting of $X/k$. Actually we could not obtain a numerical example for the easiest case when $p=3$. Nonetheless, when $p=2$, one computes the conductor exponent $f(\Jac(C)/K)$ of $\Jac(C)/K$ and its Swan conductor $\sw(\Jac(C)/K)$ :

\begin{thm}\label{prop2intro}
Under the hypotheses of Theorem \ref{prop1intro}, if $p=2$ the lower ramification filtration of $G$ is :
\[ G= G_0 = G_1 \supsetneq   {\rm Z}(G)= G_2 = \dots = G_{1+q} \supsetneq \lbrace 1 \rbrace. \]
Then, $f(\Jac(C)/K)=2q+1$ and  $\sw(\Jac(C)/\Q_2^{\rm ur})=1$.
\end{thm}

\noindent \textbf{Remarks :}
\begin{enumerate}
\item In Theorem \ref{prop1intro}, one actually obtains a family of liftings $C/K$ of $X/k$ with the announced properties. It is worth noting that there are finitely many additive polynomials $R_0(t) \in k[t]$ such that $w^p-w=tR(t)$ is $k$-isomorphic to $w^p-w=tR_0(t)$ (see \cite{LM} 8.2), so we have to solve the problem in a somehow generic way. In \cite{CM}, we obtain the analogous of Theorem \ref{prop1intro} and Theorem \ref{prop2intro} for $p \geq 2$ in the easier case $R(t)=t^q$.
\item For $p=3$, the easiest non-trivial case is such that ${[M:K]=243}$, that is why we could not even do computations using Magma to guess the behaviour of the ramification filtration of the monodromy extension for $p > 2$. Nonetheless, one shows that if $p \geq 3$, the lower ramification filtration of $G$ is
\[ G= G_0 = G_1 \supsetneq G_2 = \dots = G_{u} = {\rm Z}(G) \supsetneq \lbrace 1 \rbrace ,\]
where $u \in 1+q\N$.
\item The value $\sw(\Jac(C)/\Q_2^{\rm ur}) =1$ is the smallest one among abelian varieties over $\Q_2^{ur}$ with non tame monodromy extension. That is, in some sense, a counter part of \cite{Br-Kr} and \cite{LRS} where an upper bound for the conductor exponent is given and it is shown that this bound is actually achieved.
\end{enumerate}

\section{Background}
\label{sec:1}
\noindent \textbf{Notations.} Let $(R,v)$ be a complete discrete valuation ring (DVR) of mixed characteristic $(0,p)$ with fraction field $K$ and algebraically closed residue field $k$. We denote by $\pi_K$ a uniformizer of $R$ and assume that $K$ contains a primitive $p$-th root of unity $\zeta_p$. Let $\lambda:=\zeta_p-1$. If $\L/K$ is an algebraic extension, we will denote by $\pi_{\L}$(resp. $v_{\L}$, resp. $\L^{\circ}$) a uniformizer for $\L$ (resp. the prolongation of $v$ to  $\L$ such that $v_{\L}(\pi_{\L})=1$, resp. the ring of integers of $\L$). If there is no possible confusion we note $v$ for the prolongation of $v$ to an algebraic closure $K^{\rm  alg}$ of $K$.\\

\addtocounter{mysub}{1}
\arabic{mysub}. \textit{Stable reduction of curves.} \label{mysub:reduction} The first result is due to Deligne and Mumford (see for example \cite{Liu} for a presentation following Artin and Winters).
\begin{thm}[Stable reduction theorem]\label{stableth}
Let $C/K$ be a smooth, projective, geometrically connected curve over $K$ of genus $g(C) \geq 2$. There exists a unique finite Galois extension $\M/K$ minimal for the inclusion relation such  that $C_{\M}/\M$ has stable reduction. The stable model $\mathcal{C}$ of $C_{\M}/\M$ over $\M^{\circ}$ is unique up to isomorphism. One has a canonical injective morphism :
\begin{equation}\label{injcano}
\Gal(\M/K) \overset{i}{\hookrightarrow} \Aut_k(\mathcal{C}_k).
\end{equation}
\end{thm}

\noindent \textbf{Remarks :}\begin{enumerate}
\item Let's explain the action of $\Gal(K^{\rm alg}/K)$ on $\mathcal{C}_k/k$. The group $\Gal(K^{\rm alg}/K)$ acts on $C_{\M}:=C\times \M$ on the right. By unicity of the stable model, this action extends to $\mathcal{C}$ :

\begin{center}
\begin{tikzpicture}[scale=1]
\node (CM1) at (-1,1) {$\mathcal{C}$};
\draw[->] (-0.5,1)--(0.5,1);
\node (CM2) at (1,1) {$\mathcal{C}$};
\node (sigma1) at (-0.05,1.2) {$\sigma$};
\draw[->] (-1,0.7)--(-1,-0.5);
\draw[->] (1,0.7)--(1,-0.5);
\node (R1) at (-1,-1) {$\M^{\circ}$};
\draw[->] (-0.5,-1)--(0.5,-1);
\node (R2) at (1,-1) {$\M^{\circ}$};
\node (sigma2) at (-0.05,-0.8) {$\sigma$};
\end{tikzpicture}
\end{center}

Since $k=k^{\rm alg}$ one gets $\sigma \times k = \Id_k$, whence the announced action. The last assertion of the theorem characterizes the elements of $\Gal(K^{{\rm alg}}/\M)$ as the elements of $\Gal(K^{{\rm  alg}}/K)$ that  trivially act on $\mathcal{C}_k/k$.

\item  If $p > 2g(C)+1$, then $C/K$ has stable reduction over a tamely ramified extension of $K$. We will study examples of covers with $p \leq 2g(C)+1$.

\item Our results will cover the elliptic case. Let $E/K$ be an elliptic curve with additive reduction. If its modular invariant is integral, then there exists a smallest extension $\M$ of $K$ over which $E/K$ has good reduction. Else $E/K$ obtains split multiplicative reduction over a unique quadratic extension of $K$ ( see \cite{Kr}).
\end{enumerate}

\begin{defi}
The extension $\M/K$ is the \textsl{monodromy extension of $C/K$}. We call $\Gal(\M/K)$ the \textsl{monodromy group of $C/K$}. It has a unique $p$-Sylow subgroup $\Gal(\M/K)_1$ called the \textsl{wild monodromy group}. The extension $\M/\M^{\Gal(\M/K)_1}$ is the \textsl{wild monodromy extension}.
\end{defi}

From now on we consider smooth, projective, geometrically integral curves $C/K$ of genus $g(C) \geq 2$ birationally given by ${Y^p=f(X):=\prod_{i=0}^t (X-x_i)^{n_i}}$ with  $(p,\sum_{i=0}^tn_i)=1$, $(p,n_i)=1$ and $\forall \; 0 \leq i \leq t, x_i \in R^{\times}$. Moreover, we assume that $\forall i \neq j, \; v(x_i-x_j)=0$, that is to say, the branch locus $B= \lbrace x_0, \dots , x_t, \infty \rbrace$ of the cover has \textsl{equidistant geometry}. We denote by $Ram$ the ramification locus of the cover.\\

\noindent \textbf{Remark :} We only ask $p$-cyclic covers to satisfy Raynaud's theorem 1' \cite{Ra} condition, that is the branch locus is $K$-rational with equidistant geometry. This has consequences on the image of \eqref{injcano}.

\begin{prop}\label{propD0}
Let $\mathcal{T}={\rm Proj}( \M^{\circ}[X_0,X_1])$ with $X=X_0/X_1$. The normalization $\mathcal{Y}$ of $\mathcal{T}$ in $K(C_{\M})$ admits a blowing-up $\tilde{\mathcal{Y}}$ which is a semi-stable model of $C_{\M}/\M$. The dual graph of $\tilde{\mathcal{Y}}_k/k$ is a tree and the points in $Ram$ specialize in a unique irreducible component $D_0 \simeq \Proj_k$ of $\tilde{\mathcal{Y}}_k/k$. There exists a contraction morphism $h : \tilde{\mathcal{Y}} \to \mathcal{C}$, where $\mathcal{C}$ is the stable model of $C_{\M}/\M$ and 
\begin{equation}\label{inj}
\Gal(\M/K) \hookrightarrow \Aut_k(\mathcal{C}_k)^{\#},
\end{equation}
where $\Aut_k(\mathcal{C}_k)^{\#}$ is the subgroup of $\Aut_k(\mathcal{C}_k)$ of elements inducing the identity on $h(D_0)$.
\end{prop}

\begin{proof}see \cite{CM}.
\end{proof}

\noindent \textbf{Remark :} The component $D_0$ is the so called \textsl{original component}.

\begin{defi}
If \eqref{inj} is surjective, we say that $C$ has \textsl{maximal monodromy}. If $v_p(|\Gal(\M/K)|)=v_p(|\Aut_k(\mathcal{C}_k)^{\#}|)$, we say that $C$ has \textsl{maximal wild monodromy}.
\end{defi}

\begin{defi}
The valuation on $K(X)$ corresponding to the discrete valuation ring $R[X]_{(\pi_K)}$ is called the \textsl{Gauss valuation $v_X$} with respect to $X$. We then have
\begin{align*}
v_X \left( \sum_{i=0}^ma_iX^i \right)= \min \lbrace v(a_i), \; 0 \leq i \leq m \rbrace.
\end{align*}
Note that a change of variables $T=\frac{X-y}{\rho}$ for $y, \rho \in R$ induces a Gauss valuation $v_T$. These valuations are exactly those that come from the local rings at generic points of components in the semi-stables models of $\Proj_K$.
\end{defi}

\addtocounter{mysub}{1}
\arabic{mysub}. \textit{Extra-special $p$-groups.} The Galois groups and automorphism groups that we will have to consider are $p$-groups with peculiar group theoretic properties (see for example \cite{Hu} Kapitel III  \S 13 or \cite{Su} for an account on extra-special $p$-groups). We will denote by  ${\rm Z}(G)$ (resp. ${\rm D}(G)$, $\Phi(G)$) the center (resp. the derived subgroup,  the Frattini subgroup) of $G$. If $G$ is a $p$-group, one has  $\Phi(G)={\rm D}(G)G^p$.

\begin{defi}
An \textsl{extra-special $p$-group} is a non abelian $p$-group $G$ such that $ {\rm D}(G)={\rm Z}(G)=\Phi(G)$ has order $p$.
\end{defi}

\begin{prop}\label{lemxspegp}
Let $G$ be an extra-special $p$-group.
\begin{enumerate}
\item Then $|G|=p^{2n+1}$ for some $n \in \N^{\times}$.
\item One has the exact sequence 
\begin{align*}
0 \to {\rm Z}(G) \to G \to (\Z/p\Z)^{2n} \to 0.
\end{align*}
\item The group $G$ has an abelian subgroup $J$ such that $\Cen(G) \subseteq J$ and $|J/\Cen(G)|=p^n$.
\end{enumerate}
\end{prop}

%

\addtocounter{mysub}{1}
\arabic{mysub}. \textit{Galois extensions of complete DVRs.} Let $\L/K$ be a finite Galois extension with group $G$. Then $G$ is endowed with a \textsl{lower ramification filtration} $(G_i)_{i \geq -1}$ where $G_i$ is the \textsl{$i$-th lower ramification group} defined by $G_i := \lbrace \sigma \in G \; | \; v_{\L}(\sigma(\pi_{\L})-\pi_{\L}) \geq i+1 \rbrace$. The integers $i$ such that $G_i \neq G_{i+1}$ are called \textsl{lower breaks}. For $\sigma \in G- \lbrace 1 \rbrace$, let $i_G(\sigma):= v_{\L}(\sigma(\pi_{\L})-\pi_{\L})$. The group $G$ is also endowed with a \textsl{higher ramification filtration} $(G^i)_{i \geq -1}$ which can be computed from the $G_i$'s by means of the \textsl{Herbrand's function} $\varphi_{\L/K}$. The real numbers $t$ such that $ \forall \epsilon >0, \; G^{t+\epsilon} \neq G^t$  are called \textsl{higher breaks}.\\

\begin{lem}\label{maurizio}
Let $M/K$ be a Galois extension such that $\Gal(M/K)$ is an extra-special $p$-group of order $p^{2n+1}$. Assume that $\Gal(M^{\Cen(G)}/K)_2= \lbrace 1 \rbrace$, then the break $t$ of $M/M^{\Cen(G)}$  is such that $t \in 1+p^n\N$.
\end{lem}
\begin{proof}
According to  Proposition \ref{lemxspegp} \textit{3.}, there exists an abelian subgroup $J$ with $\Cen(G) \subseteq J \subseteq G$ and $|J/\Cen(G)|=p^n$. Thus, one has the following diagram 

\begin{center}
\begin{tikzpicture}[scale=2]
\node (K) at (-1,-1) {$K$};
\node(d1) at (-1.6,0.5) {$[M:L]=p$};
\node(d2) at (-1.6,-0.5) {$[L:K]=p^{2n}$};
\node(d3) at (0.1,-0.95) {$[F:K]=p^n$};
\node (L) at (-1,0) {$L:=M^{\Cen(G)}$};
\node (M) at (-1,1) {$M$};
\node (F) at (0.5,-0.5) {$F:=M^J$};
\draw[-] (-1,-0.8)--(-1,-0.2);
\draw[-] (-1,0.2)--(-1,0.8);
\draw[-] (-0.9,-1)--(0,-0.6);
\draw[-] (0,-0.45)--(-0.75,-0.05);
\draw[-] (0,-0.3)--(-0.9,0.9);
\end{tikzpicture}
\end{center}
Let $t$ be the lower break of $M/L$, then $t$ is a lower break of $M/F$ and $\varphi_{M/F}(t)=\varphi_{L/F}(\varphi_{M/L}(t))$ is a higher break of $M/F$. Since $\varphi_{M/L}(t)=t$, one has $\varphi_{M/F}(t)=\varphi_{L/F}(t)$. Since $\Gal(L/K)_2= \lbrace 1 \rbrace$, one has $\Gal(L/F)_2= \lbrace 1 \rbrace$ and $\varphi_{L/F}(t)=1+\frac{t-1}{p^n}$. The Hasse-Arf Theorem applied to the abelian extension $M/F$ implies that $1+\frac{t-1}{p^n}\in \N - \lbrace 0 \rbrace$, thus $t \in 1 + p^n\N$.

\end{proof}

\addtocounter{mysub}{1}
\arabic{mysub}. \textit{Torsion points on abelian varieties.} Let $A/K$ be an abelian variety over $K$ with  potential good reduction and $ \ell \neq p$ be a prime number. We denote by $A[\ell]$ the $\ell$-torsion group of $A(K^{\rm  alg})$ and by $T_{\ell}(A)= \varprojlim A[\ell^n]$ (resp. $V_{\ell}(A)=T_{\ell}(A) \otimes\Q_{\ell}$) the Tate module (resp. $\ell$-adic Tate module) of $A$.

\indent The following result may be found in \cite{Gur} (paragraph 3). We recall it for the convenience of the reader.

\begin{lem}\label{guralnick}
Let $k=k^{\rm alg}$  be a field with ${\rm char} \; k=p \geq 0$ and $C/k$ be a projective, smooth, integral curve. Let $\ell \neq p$ be a prime number and $H$ be a finite subgroup of $\Aut_k(C)$ such that $(|H|,\ell)=1$. Then 
\[2g(C/H)= \dim_{\F_{\ell}} \Jac(C)[\ell]^H.\]
\end{lem}

\indent If $\ell \geq 3$, then  $\L=K(A[\ell])$ is the minimal extension over which $A/K$ has good reduction. It is a Galois extension with group $G$ (see \cite{SerTat}). We denote by $r_G$ (resp. $1_G$) the character of the regular (resp. unit) representation of $G$. We denote by $I$ the inertia group of $K^{\rm alg}/K$. For further explanations about  conductor exponents see \cite{Ser2}, \cite{Ogg} and \cite{SerTat}. 
\begin{defi}
\begin{enumerate}
\item Let 
\begin{align*}
a_G(\sigma) & := -i_G(\sigma), \; \; \; \;\sigma \neq 1,\\
a_G(1)      & := \sum_{\sigma \neq 1} i_G(\sigma),
\end{align*}
 and $\sw_G:=a_G-r_G+1_G$. Then, $a_G$ is the character of a $\Q_{\ell}[G]$-module and there exists a projective $\Z_{\ell}[G]$-module $\Sw_G$ such that $\Sw_G \otimes_{\Z_{\ell}} \Q_{\ell}$ has character $\sw_G$.
\item We still denote by $T_{\ell}(A)$ (resp. $A[\ell]$) the $\Z_{\ell}[G]$-module (resp. $\F_{\ell}[G]$-module) afforded by $G \to \Aut(T_{\ell}(A))$ (resp. $ G \to \Aut(A[\ell])$). Let 
\begin{align*}
\sw(A/K)&:=\dim_{\F_{\ell}} \Hom_G(\Sw_G,A[\ell]),\\
\epsilon(A/K)&:=\codim_{\Q_{\ell}} V_{\ell}(A)^I.
\end{align*}
The integer $f(A/K):=\epsilon(A/K)+\sw(A/K)$ is the so called \textsl{conductor exponent of $A/K$} and $\sw(A/K)$ is the \textsl{Swan conductor of $A/K$}.
\end{enumerate}
\end{defi}

\begin{prop}Let $\ell \neq p$, $ \ell \geq 3$ be a prime number.
\begin{enumerate} 
\item The integers $\sw(A/K)$ and $\epsilon(A/K)$ are independent of $\ell$.
\item One has 
\[ \sw(A/K) = \sum_{i \geq 1} \frac{|G_i|}{|G_0|} \dim_{\F_{\ell}} A[\ell]/A[\ell]^{G_i}.\]
Moreover, for $\ell$ large enough, $\epsilon(A/K)=\dim_{\F_{\ell}}A[\ell]/A[\ell]^{G_0}$.
\end{enumerate}
\end{prop}
\noindent \textbf{Remark :} It follows from the definition that $\sw(A/K)=0$ if and only if $G_1= \lbrace 1 \rbrace$. The Swan conductor is a measure of the wild ramification.\\

\addtocounter{mysub}{1}
\arabic{mysub}. \textit{Automorphisms of Artin-Schreier covers.} See \cite{LM} for further results on this topic. Let $R(t) \in k[t]$ be a monic additive polynomial and $A_R/k$ be the  smooth, projective, geometrically irreducible curve birationally given by $w^p-w=tR(t)$. There is a so called Artin-Schreier morphism $\pi : A_R \to \Proj_k$. The automorphism $t_a$ of $\Proj_k$  given by $t \mapsto t+a$ with $a \in k$ has a prolongation $\tilde{t}_a$ to $A_R$ if there is a commutative diagram 

\begin{center}
\begin{tikzpicture}[scale=1]
\node (CM1) at (-1,1) {$A_R$};
\draw[->] (-0.5,1)--(0.5,1);
\node (CM2) at (1,1) {$A_R$};
\node (sigma1) at (-0.05,1.25) {$\tilde{t}_a$};
\draw[->] (-1,0.7)--(-1,-0.5);
\node (pi1) at (-1.2,0) {$\pi$};
\node (pi2) at (0.8,0) {$\pi$};
\draw[->] (1,0.7)--(1,-0.5);
\node (R1) at (-1,-1) {$ \Proj_k$};
\draw[->] (-0.5,-1)--(0.5,-1);
\node (R2) at (1,-1) {$\Proj_k$};
\node (sigma2) at (-0.05,-0.8) {$t_a$};
\end{tikzpicture}
\end{center}

\begin{prop}\label{lem:prolon}
Let $n \geq 1$, $q:=p^n$ and $R(t):=\sum_{k=0}^{n-1}\bar{u}_kt^{p^k} + t^q\in k[t]$. The automorphism of $\Proj_k$  given by $t \mapsto t+a$ with $a \in k$ has a prolongation to $A_R/k$ if and only if one has 
\[a^{q^2}+(2\bar{u}_0a)^q+\sum_{k=1}^{n-1}(\bar{u}_k^qa^{qp^k}+(\bar{u}_ka)^{q/p^k})+a=0.\]
\end{prop}

\section{Main theorem}
\label{sec:2}
We start by fixing notations that will be used throughout this section.\\

\noindent \textbf{Notations.} We denote by $\mathfrak{m}$ the maximal ideal of $ (K^{\rm alg})^{\circ}$. Let $n \in \N^{\times}$, ${q:=p^n}$, $a_n:=(-1)^q(-p)^{p+p^2+\dots+q}$ and $\forall \; 0 \leq i \leq n-1$, ${d_i:=p^{n-i+1}+\dots+q}$. We denote by $\Qpur$ the maximal unramified extension of $\Q_p$ and we put ${K:=\Qpur(\lambda^{1/(1+q)})}$. Let $\underline{\rho}:=(\rho_0, \dots, \rho_{n-1})$ where $ \forall \; 0 \leq k \leq n-1$, $\rho_k \in K$, $\rho_k=u_k\lambda^{p(q-p^k)/(1+q)}$ and $v(u_k)=0$ or $u_k=0$. For $c \in R$, let 
\begin{align*}
&f_{c,\rho}(X):=1+\sum_{k=0}^{n-1}\rho_kX^{1+p^k}+cX^q+X^{1+q},\\
{\rm and} \; & s_{1,\rho}(X):=2\rho_0X+\sum_{k=1}^{n-1}\rho_kX^{p^k}+X^q.
\end{align*}
One defines the \textsl{modified monodromy polynomial} $L_{c,\rho}(X)$ by 
\[s_{1,\rho}(X)^q-a_nf_{c,\rho}(X)^{q-1}(c+X)-(-1)^q\sum_{k=1}^{n-1}(\rho_kX)^{q/p^k}(-p)^{d_k}f_{c,\rho}(X)^{q(p^k-1)/p^k}.\]
Let $C_{c,\rho} / K $ and $A_u/k$ be the smooth projective integral curves birationally given respectively by $Y^p=f_{c,\rho}(X)$ and $w^p-w=\sum_{k=0}^{n-1}\bar{u}_kt^{1+p^k}+t^{1+q}$.

\begin{thm}\label{maintheo}
The curve $C_{c,\rho}/K$ has potential good reduction isomorphic to $A_u/k$.
\begin{enumerate}
\item If $v(c) \geq v(\lambda^{p/(1+q)})$, then the monodromy extension of $C_{c,\rho}/K$ is trivial.
\item If $v(c) < v(\lambda^{p/(1+q)})$, let $y$ be a root of $L_{c,\rho}(X)$ in $K^{\rm  alg}$. Then $C_{c,\rho}$ has good reduction over $K(y,f_{c,\rho}(y)^{1/p})$. If $L_{c,\rho}(X)$ is irreducible over $K$, then $C_{c,\rho}/K$ has maximal wild monodromy. The monodromy extension of $C_{c,\rho}/K$ is $\M=K(y,f_{c,\rho}(y)^{1/p})$ and $G=\Gal(\M/K)$ is an extra-special $p$-group of order $pq^2$. 
\item Assume that $c=1$. The polynomial $L_{1,\rho}(X)$ is irreducible over $K$. The lower ramification filtration of $G$ is
\[G= G_0 = G_1 \supsetneq G_2 = \dots = G_{u} = {\rm Z}(G) \supsetneq \lbrace 1 \rbrace , \]
 with $u \in 1+q\N$. Moreover, if $p=2$, then $u=1+q$, one has $f(\Jac(C_{1,\rho})/K)=2q+1$ and $\sw(\Jac(C_{1,\rho})/\Q_2^{\rm ur})=1$.
\end{enumerate}
\end{thm}

\begin{proof}
\textit{1.} Assume that $v(c) \geq v(\lambda^{p/(1+q)})$. Set $\lambda^{p/(1+q)}T=X$ and ${\lambda W+1=Y}$. Then, the equation defining $C_{c,\rho}/K$ becomes
\[ (\lambda W+1)^p=1+\sum_{k=0}^{n-1}\rho_k\lambda^{p(1+p^k)/(1+q)}T^{1+p^k}+c\lambda^{pq/(1+q)}T^q+\lambda^pT^{1+q}. \]
After simplification by $\lambda^p$ and reduction modulo $\pi_K$ this equation gives :
\begin{equation}\label{equa:1}
 w^p-w=\sum_{k=0}^{n-1}\bar{u}_kt^{1+p^k}+at^q+t^{1+q}, \; a \in k.
\end{equation}
By Hurwitz formula the genus of the curve defined by \eqref{equa:1} is seen to be that of $C_{c,\rho}/K$ . Applying  \cite{Liu} 10.3.44, there is a component in the stable reduction birationally given by \eqref{equa:1}. The stable reduction being a tree, the curve $C_{c,\rho}/K$ has good reduction over $K$.\\

\noindent \textit{2.} The proof is divided into eight steps. Let $y$ be a root of $L_{c,\rho}(X)$. \\

\noindent \textbf{Step I :} \textsl{One has $v(y)=v(a_nc)/q^2$.}

\noindent By expanding $L_{c,\rho}(X)$, one shows that its Newton polygon has a single slope $v(a_nc)/q^2$. The polynomial  $L_{c,\rho}(X)$ has degree $q^2$ and its leading (resp. constant) coefficient has valuation $0$ (resp. $v(a_nc)$).
One examines monomials from $a_nf_{c,\rho}^{q-1}(X)(c+X)$. Since $v(c) < v(\lambda^{p/(1+q)})$, one checks that 
\[\forall 1 \leq i \leq q^2-1, \; \frac{v(a_n)}{q^2-i} \geq \frac{v(a_nc)}{q^2}. \] Then one examines monomials from $(\rho_iX)^{q/p^i}p^{d_i}f_{c,\rho}(X)^{q(p^i-1)/p^i}$. They have degree  at least $q/p^i$, thus one checks that 
\[\forall \; 1 \leq i \leq n-1, \; \frac{q/p^iv(\rho_i)+d_iv(p)}{q^2-q/p^i} \geq  \frac{v(a_nc)}{q^2}. \]
The monomial $X^{q^2}$ in $s_{1,\rho}(X)^q$ corresponds to the point $(0,0)$ in the Newton polygon of  $L_{c,\rho}(X)$, the other monomials of $s_{1,\rho}(X)^q$ produce a slope greater than $v(\rho_i)/(q-p^i)$ and one checks that
\[ \forall \; 0 \leq i \leq n-1,\; \frac{v(\rho_i)}{q-p ^i} \geq \frac{v(a_nc)}{q^2}. \]
Note that \textbf{Step I} implies that $v(f_{c,\rho}(y))=0$, we will use this remark throughout this proof.\\
 
\noindent \textbf{Step II :} \textsl{Define $S$ and $T$ by $\lambda^{p/(1+q)}T=(X-y)=S$. Then $f_{c,\rho}(S+y)$ is congruent modulo $\lambda^p \mathfrak{m}[T]$ to}
\begin{align*}
 f_{c,\rho}(y)+s_{1,\rho}(y)S+\sum_{k=0}^{n-1}\rho_kS^{1+p^k}+\sum_{k=1}^{n-1}\rho_kyS^{p^k}+(c+y)S^q+S^{1+q} .
\end{align*}
Using the following formula for $A \in K^{\rm alg}$ with $v(A)>0$ and $B \in (K^{\rm alg})^{\circ}[T]$
\begin{equation*}
k\geq 1,\; (A+B)^{p^k} \equiv (A^{p^{k-1}}+B^{p^{k-1}})^p \md p^2 \mathfrak{m}[T], 
\end{equation*}
one computes $\md \lambda^p \mathfrak{m}[T]$
\[f_{c,\rho}(y+S) = 1+\sum_{k=0}^{n-1}\rho_k(y+S)^{1+p^k}+(y+S)^{1+q}+c(y+S)^q\]
\[\equiv 1+\rho_0(y+S)^2+\sum_{k=1}^{n-1}\rho_k(y+S)(y^{p^{k-1}}+S^{p^{k-1}})^p+(y+S+c)(y^{q/p}+S^{q/p})^p.\]
Using \textbf{Step I}, one checks that for all $1 \leq k \leq n-1$ 
\[ \rho_k(y^{p^{k-1}}+S^{p^{k-1}})^p \equiv \rho_k(y^{p^k}+S^{p^k}) \md \lambda^p \mathfrak{m}[T], \]
and $(y^{q/p}+S^{q/p})^p \equiv y^q+S^q \md \lambda^p \mathfrak{m}[T]$. It follows that 
\[f_{c,\rho}(y+S) \equiv 1+\rho_0(y+S)^2+\sum_{k=1}^{n-1}\rho_k(y+S)(y^{p^k}+S^{p^k})+(y+c+S)(y^q+S^q).\]
One easily concludes from this last expression.\\

\noindent \textbf{Step III :} \textsl{Let $R_1:=K[y]^{\circ}$. For all $0 \leq i \leq n$, one defines $A_i(S) \in R_1[S]$ and $B_i \in R_1$ by induction :}
\[B_n:= -s_{1,\rho}(y), \;\;\forall \; 1 \leq i \leq n-1,  \;B_i:=  \frac{f_{q,c}(y)B_{i+1}^p}{(-pf_{c,\rho}(y))^p}-y\rho_{n-i},\; \;  \]
\[ {\rm and} \; B_0:=\frac{f_{c,\rho}(y)B_1^p}{(-pf_{c,\rho}(y))^p},\]
\[\; \; \; A_0(S):=\;  0 \;  {\rm and} \;\forall\; 0 \leq i \leq n-1\; \; SA_{i+1}(S):= SA_i(S)-\frac{B_{i+1}S^{q/p^{i+1}}}{pf_{q,c}(y)^{(p-1)/p}}.\]
\textsl{Then for all $0\leq i \leq n-1$, $v(B_{i+1})= (1+\dots+p^i)v(p)/p^i+v(c)/p^{i+1}$ and modulo $\lambda^{\frac{pq^2}{q+1}} \mathfrak{m}$ one has }
\begin{equation}\label{eqStepIII}
B_n^q \equiv \frac{a_n}{(-1)^q}f_{c,\rho}(y)^{q-1}B_0+\sum_{k=1}^{n-1}(\rho_ky)^{q/p^k}(-p)^{d_k}f_{c,\rho}(y)^{q(p^k-1)/p^k}.
\end{equation}
\noindent We prove the claim  about $v(B_{i+1})$ by induction on $i$. Using \textbf{Step I}, one checks that $\forall  \; 0 \leq k \leq n-1, \; v(\rho_ky^{p^k}) > v(y^q)$, so $v(B_n)=v(y^q)$. Assume that we have shown the claim for $i$, then one checks that $v((B_{i+1}/p)^p) < v(y\rho_{n-i})$ and one deduces $v(B_i)$ from the definition of $B_i$. According to the expression of $v(B_i)$, one has $\forall \; 0 \leq i \leq n, A_i(S) \in R_1[S]$.

Then we prove the second part of \textbf{Step III}. From the definition of the $B_i$'s one obtains that for all $ 1 \leq i \leq n-1$ 
\begin{equation}\label{equa:2}
B_{n-i+1}^{q/p^{i-1}} = (-p)^{q/p^{i-1}}f_{c,\rho}(y)^{q(p-1)/p^i}(y\rho_i+B_{n-i}(y))^{q/p^i}.
\end{equation}
Using \textbf{Step I} and $v(B_{n-1})$ one checks that $\forall \; 1 \leq k \leq q/p-1$
\[p^q\binom{q/p}{k}(y\rho_1)^kB_{n-1}^{q/p-k} \equiv 0  \md \lambda^{pq^2/(1+q)}\mathfrak{m} ,\]
so $ p^q(y\rho_1+B_{n-1})^{q/p} \equiv p^q((y\rho_1)^{q/p}+B_{n-1}^{q/p})\md \lambda^{pq^2/(1+q)}\mathfrak{m}$. Thus, applying equation \eqref{equa:2} with $i=1$, one gets 
\begin{align*}
B_n^q &= (-p)^qf_{c,\rho}(y)^{q(p-1)/p}(y\rho_1+B_{n-1})^{q/p}\\
      &\equiv (-p)^qf_{c,\rho}(y)^{q(p-1)/p}((y\rho_1)^{q/p}+B_{n-1}^{q/p}) \md \lambda^{pq^2/(1+q)}\mathfrak{m}.
\end{align*} 
One checks using \textbf{Step I} and $v(B_{n-i})$  that  $\forall \; 1 \leq i \leq n-1$ and $1 \leq k \leq q/p^i-1$
\[ p^{q+\dots+q/p^{i-1}}\binom{q/p^i}{k}B_{n-i}^{q/p^i-k}(y\rho_i)^k \equiv 0 \md \lambda^{pq^2/(1+q)}\mathfrak{m}, \]
then by induction on $i$, using equation \eqref{equa:2}, one shows that modulo $\lambda^{pq^2/(1+q)}\mathfrak{m}$
\begin{equation*}
B_n^q\equiv (-p)^{p+\dots+q}f_{c,\rho}(y)^{q-1}B_0+ \sum_{k=1}^{n-1}(\rho_ky)^{q/p^k}(-p)^{d_k}f_{c,\rho}(y)^{q(p^k-1)/p^k}.
\end{equation*}
\noindent \textbf{Step IV :} \textsl{One has modulo $\lambda^p\mathfrak{m}[T]$ }
\begin{equation*}
f_{c,\rho}(S+y) \equiv f_{c,\rho}(y)+s_{1,\rho}(y)S+\sum_{k=0}^{n-1}\rho_kS^{1+p^k}+\sum_{k=1}^{n-1}y\rho_kS^{p^k}+B_0S^q+S^{1+q}.
\end{equation*}

\noindent 
Since $L_{c,\rho}(y)=0$, one has
\begin{equation}\label{equa:3}
s_{1,\rho}(y)^q=a_nf_{c,\rho}(y)^{q-1}(c+y)+(-1)^q\sum_{k=1}^{n-1}(\rho_ky)^{q/p^k}(-p)^{d_k}f_{c,\rho}(y)^{q(p^k-1)/p^k}.
\end{equation}
Using $B_n:=-s_{1,\rho}(y)$, equations \eqref{eqStepIII} and \eqref{equa:3} one gets
\begin{equation*}
a_nf_{c,\rho}(y)^{q-1}(c+y-B_0) \equiv 0 \md \lambda^{pq^2/(q+1)} \mathfrak{m}.
\end{equation*}
which is equivalent to $S^q(y+c-B_0) \equiv 0 \md  \lambda^p\mathfrak{m}[T]$.
Then, \textbf{Step IV} follows from \textbf{Step II}.\\

\noindent \textbf{Step V :} \textsl{One has }
\begin{equation*}
f_{c,\rho}(S+y) \equiv (f_{c,\rho}(y)^{1/p}+SA_n(S))^p+\sum_{k=0}^{n-1}\rho_kS^{1+p^k}+S^{1+q} \md \lambda^p\mathfrak{m}[T].
\end{equation*}
Let $R:=\sum_{k=0}^{n-1}\rho_kS^{1+p^k}+S^{1+q}+s_{1,\rho}(y)S$. Since $B_n=-s_{1,\rho}(y)$ one has
\begin{align}\label{equa:4}
&(f_{c,\rho}(y)^{1/p}+SA_n(S))^p+\sum_{k=0}^{n-1}\rho_kS^{1+p^k}+S^{1+q} \nonumber\\
= \; &(f_{c,\rho}(y)^{1/p}+SA_n(S))^p + B_nS+ R \nonumber\\
= \; &\Big(f_{c,\rho}(y)^{1/p}+SA_{n-1}(S)-\frac{B_nS}{pf_{q,c}(y)^{(p-1)/p}}\Big)^p+B_nS+R\nonumber\\
= \;  &(f_{c,\rho}(y)^{1/p}+SA_{n-1}(S))^p+\Big(\frac{-B_nS}{pf_{q,c}(y)^{(p-1)/p}}\Big)^p+B_nS+R+\Sigma,
\end{align}
where 
\begin{equation}\label{equa:5}
\Sigma = \sum_{k=1}^{p-1}\binom{p}{k}(f_{c,\rho}(y)^{1/p}+SA_{n-1}(S))^{p-k}\Big(\frac{-B_nS}{pf_{q,c}(y)^{(p-1)/p}}\Big)^k.
\end{equation}
Using the expression of $v(B_n)$ computed in \textbf{Step III}, one checks that the terms with $k \geq 2$ in \eqref{equa:5} are zero modulo $\lambda^p\mathfrak{m}[T]$. It implies the following relations
\begin{align*}
\Sigma+ B_nS \equiv \;   & B_nS\bigg[1-\frac{(f_{c,\rho}(y)^{1/p}+SA_{n-1}(S))^{p-1}}{f_{c,\rho}(y)^{(p-1)/p}}\bigg]\\
\equiv \; & \frac{B_nS}{f_{c,\rho}(y)^{(p-1)/p}} \Big[f_{c,\rho}(y)^{(p-1)/p}-(f_{c,\rho}(y)^{1/p}+SA_{n-1}(S))^{p-1}\Big]\\
\equiv \; & \frac{B_nS}{f_{c,\rho}(y)^{(p-1)/p}}\Big[-\sum_{k=1}^{p-1}\binom{p-1}{k}f_{c,\rho}(y)^{(p-1-k)/p}(SA_{n-1}(S)^k\Big]\\
\equiv \;&0 \md \lambda^p\mathfrak{m}[T],\; {\rm since \;  for \; } k \geq 1, \; B_nS^{k+1} \equiv 0 \md \lambda^p\mathfrak{m}[T].
\end{align*}
According to the definition of $B_{n-1}$ (see \textbf{Step III}) one obtains
\begin{equation}\label{equa:6}
\eqref{equa:4} \equiv (f_{c,\rho}(y)^{1/p}+SA_{n-1}(S))^p + R + B_{n-1}S^p+y\rho_1S^p \md \lambda^p\mathfrak{m}[T].
\end{equation}
Using the same process, one shows by induction on $i$ that \eqref{equa:4} is congruent to
\begin{equation}\label{equa:7}
(f_{c,\rho}(y)^{1/p}+SA_{i+1}(S))^p + B_{i+1}S^{p^{n-i-1}}+\sum_{k=1}^{n-i-1}y\rho_kS^{p^k}  + R \; \md \lambda^p\mathfrak{m}[T].
\end{equation}

Thus, one applies equation \eqref{equa:7} with $i=0$ 
\[ \eqref{equa:4} \equiv (f_{c,\rho}(y)^{1/p}+SA_{1}(S))^p + B_{1}S^{q/p}+\sum_{k=1}^{n-1}y\rho_kS^{p^k}  + R \; \md \lambda^p\mathfrak{m}[T].\]
One defines $\Sigma'$ by $(f_{c,\rho}(y)^{1/p}+SA_{1}(S))^p = f_{c,\rho}(y)+(SA_{1}(S))^p+\Sigma'$. From $pf_{c,\rho}(y)^{(p-1)/p}SA_1(S)=-B_1S^{q/p}$ (see the definition of $SA_1(S)$) one gets  
\[\Sigma'+B_1S^{q/p}=\sum_{k=2}^{p-1}\binom{p}{k}f_{c,\rho}(y)^{(p-k)/p}(SA_1(S))^k, \]
so using the expression of $v(B_1)$ computed in \textbf{Step III}, one checks that 
$\Sigma'+B_1S^{q/p} \equiv 0 \md \lambda^p\mathfrak{m}[T]$. From the definition of $SA_1(S)$ and $B_0$ one has $(SA_1(S))^p=B_0S^q$, thus
\[\eqref{equa:4}  \equiv f_{c,\rho}(y)+B_0S^q+\sum_{k=1}^{n-1}y\rho_kS^{p^k}  + R  \md \lambda^p\mathfrak{m}[T]. \]
Then, \textbf{Step V} follows from \textbf{Step IV} and this last relation.\\

\noindent \textbf{Step VI :} \textsl{The curve $C_{c,\rho}/K$ has good reduction over $K(y,f_{c,\rho}(y)^{1/p})$.}

\noindent According to \textbf{Step V}, the change of variables in $K(y,f_{c,\rho}(y)^{1/p})$
\begin{align*}
X=\lambda^{p/(1+q)}T+y=S+y \; \;{\rm and} \; \; Y=\lambda W+f_{c,\rho}(y)^{1/p}+SA_n(S),
\end{align*}
induces in reduction $w^p-w=\sum_{k=0}^{n-1}\bar{u}_kt^{1+p^k}+t^{1+q}$ with genus $g(C_{c,\rho})$. So \cite{Liu} 10.3.44 implies that this change of variables gives the stable model. Note that the $\rho_k$'s were chosen to obtain this equation for the special fiber of the stable model.\\

\noindent \textbf{Step VII :} \textsl{For any distinct roots $y_i$, $y_j$ of $L_{c,\rho}(X)$, $v(y_i-y_j)=v(\lambda^{p/(1+q)})$.}

\noindent The changes of variables $\lambda^{p/(1+q)}T=X-y_i$ and $\lambda^{p/(1+q)}T=X-y_j$ induce equivalent Gauss valuations of $K(C_{c,\rho})$, else applying \cite{Liu} 10.3.44 would contradict the uniqueness of the stable model. Thus ${v(y_i-y_j)} \geq v(\lambda^{p/(1+q)})$.

One checks that $v(f_{c,\rho}'(y))>0$, $\forall \; 0 \leq k \leq n-1$ $v(\rho_k^{q/p^k}p^{d_k}q/p^k)>v(a_n)$, ${v(s_{1,\rho}'(y))>0}$, $v(s_{1,\rho}(y))=v(y^q)$ and $v(qs_{1,\rho}(y)^{q-1}s_{1,\rho}'(y))>v(a_n)$, so 
\[ v(L_{c,\rho}'(y))=v(a_n)=(q^2-1)v(\lambda^{p/(1+q)}). \]
Taking into account that $L_{c,\rho}'(y_i)=\prod_{j \neq i}(y_i-y_j)$ and $\deg L_{c,\rho}(X)=q^2$, one obtains $v(y_i-y_j) = v(\lambda^{p/(1+q)})$.\\

\noindent \textbf{Step VIII :} \textsl{If $L_{c,\rho}(X)$ is irreducible over $K$, then $K(y,f_{c,\rho}(y)^{1/p})$ is the monodromy extension $\M$ of $C_{c,\rho}/K$ and $G:=\Gal(\M/K)$ is an extra-special $p$-group of order $pq^2$.}

\noindent Let $(y_i)_{i=1,\dots,q^2}$ be the roots of $L_{c,\rho}(X)$, $\L:=K(y_1, \dots, y_{q^2})$ and  $\M/K$ be the monodromy extension of $C_{c,\rho}/K$. Any  $\tau \in \Gal(\L/K) -\lbrace 1 \rbrace$ is such that $\tau(y_i)=y_j$ for some $i \neq j$. Thus, the change of variables 
\begin{align*}
X=\lambda^{p/(1+q)}T+y_i \; \;{\rm and} \; \;  Y=\lambda W+f_{c,\rho}(y)^{1/p}+SA_n(S),
\end{align*}
 induces the stable model and $\tau$ acts on it by :
\begin{equation*}
\tau(T)=\frac{X-y_j}{\lambda^{p/(1+q)}}, \; \; \;  {\rm hence} \; \; \;  T-\tau(T)=\frac{y_j-y_i}{\lambda^{p/(1+q)}}.
\end{equation*}
According to \textbf{Step VII}, $\tau$ acts non-trivially on the stable reduction. It follows that $\L \subseteq \M$. Indeed if $\Gal(K^{\rm alg}/\M) \nsubseteq \Gal(K^{\rm alg}/\L)$ it would exist $ {\sigma \in \Gal(K^{\rm alg}/\M)}$ inducing $\bar{\sigma} \neq \Id \in \Gal(\L/K)$, which would contradict the characterization of $\Gal(K^{\rm alg}/\M)$ (see remark after Theorem \ref{stableth}) . \\
\indent According to \cite{LM}, the $p$-Sylow subgroup $\Aut_k(\mathcal{C}_k)_1^{\#}$ of $\Aut_k(\mathcal{C}_k)^{\#}$ is an extra-special $p$-group of order $pq^2$. Moreover, one has  :
\begin{align*}
0 \to {\rm Z}(\Aut_k(\mathcal{C}_k)_1^{\#})\to \Aut_k(\mathcal{C}_k)_1^{\#} \to (\Z/p\Z)^{2n} \to 0,
\end{align*}
where $(\Z/p\Z)^{2n}$ is identified with the group of translations $t \mapsto t+a$ extending to elements of $\Aut_k(\mathcal{C}_k)_1^{\#}$. Therefore we have morphisms 
\begin{equation*}
 \Gal(\M/K) \overset{i}{\hookrightarrow} \Aut_k(\mathcal{C}_k)_1^{\#} \overset{\varphi}{\to} \Aut_k(\mathcal{C}_k)_1^{\#} /{\rm Z}(\Aut_k(\mathcal{C}_k)_1^{\#}).
\end{equation*}
The composition is seen to be surjective since the image contains the $q^2$ translations 
$t \mapsto t+ \overline{(y_i-y_1)/\lambda^{p/(1+q)}}$. Consequently, $i(\Gal(\M/K))$ is a subgroup of $\Aut_k(\mathcal{C}_k)_1^{\#}$ of index at most $p$. So it contains $\Phi(\Aut_k(\mathcal{C}_k)_1^{\#})={\rm Z}(\Aut_k(\mathcal{C}_k)_1^{\#})=\Ker \varphi$. It implies that $i$ is an isomorphism and ${\left[\M:K\right]=pq^2}$. By \textbf{Step VI}, one has $\M \subseteq K(y,f_{q,c}(y)^{1/p})$, hence $\M = K(y,f_{q,c}(y)^{1/p})$.\\

We show that $K(y_1)/K$ is Galois and that $\Gal(\M/K(y_1)) ={\rm Z}(G)$. Indeed, $\M/K(y_1)$ is $p$-cyclic and generated by $\sigma$ defined by :
\[ \sigma(y_1)=y_1 \;  {\rm and} \; \sigma(f_{c,\rho}(y_1)^{1/p})=\zeta_pf_{c,\rho}(y_1)^{1/p}. \]
According to \textbf{Step VI}, $\sigma$ acts on the stable model by :
\begin{equation*}
\sigma(S)=S, \; \; \;  \sigma(Y)=Y=\lambda \sigma(W)+\zeta_pf_{c,\rho}(y_1)^{1/p}+SA_n(S).
\end{equation*}
Hence 
\begin{equation*} 
\sigma(W)= W-f_{c,\rho}(y_1)^{1/p}.
\end{equation*}
It follows that, in reduction, $\sigma$ induces a morphism that generates ${\rm Z}(\Aut_k(\mathcal{C}_k)_1^{\#})$. It implies that $K(y_1)/K$ is Galois, $\Gal(\M/K(y_1)) ={\rm Z}(G)$ and $\Gal(K(y_1)/K)  \simeq (\Z/p\Z)^{2n}$.\\

\noindent \textit{3.} Let $L_{\rho}(X):=L_{1,\rho}(X)$, $f_{\rho}(X):=f_{1,\rho}(X)$, $s_{\rho}(y):=s_{1,\rho}(y)$, $y$ be a root of $L_{\rho}(X)$ and ${b_n:=(-1)(-p)^{1+p+\dots+p^{n-1}}}$. Note that  $b_n^p=a_n$, $L=K(y)$ and we do not assume $p=2$ until \textbf{Step E}.\\

\noindent \textbf{Step A :} \textsl{The polynomial $L_{\rho}(X)$ is irreducible over $K$.}

\noindent Let $\tilde{s}:=s_{\rho}(y)-y^q$, $\sigma:=\sum_{k=1}^q\binom{q}{k}\tilde{s}^ky^{q(q-k)}$ and $R_1:=\sum_{k=1}^{p-1}\binom{p}{k}y^{kq^2/p}(-b_n)^{p-k}$. Since $L_{\rho}(y)=0$ one has
\begin{equation*}
y^{q^2}+\sigma=s_{\rho}(y)^q =a_nf_{\rho}(y)(1+y)+\sum_{k=1}^{n-1}(\rho_ky)^{q/p^k}(-p)^{d_k}(-1)^qf_{\rho}(y)^{q(p^k-1)/p^k}.
\end{equation*}
It implies that $(y^{q^2/p}-b_n)^p$ equals
\begin{equation*}
 a_n \big[f_{\rho}(y)(1+y)+(-1)^p\big]+\sum_{k=1}^{n-1}(\rho_ky)^{q/p^k}(-p)^{d_k}(-1)^qf_{\rho}(y)^{q(p^k-1)/p^k}+R_1-\sigma.\\
\end{equation*}
We are going to remove monomials with valuation greater than $v(a_ny)$ in the above expression by taking $p$-th roots. Note that if  $\forall i \geq 1$, $\rho_i=0 $, then one could skip most of \textbf{Step A} (see equation \eqref{exp:t}). Assume that $\rho_i \neq 0$ for some $i \geq 1$, let $j:=\max \lbrace 1 \leq i \leq n-1, \rho_i \neq 0 \rbrace $ and $l:=\min \lbrace 1 \leq i \leq n-1, \rho_i \neq 0 \rbrace $. The following relations are straight forward computations using \textbf{Step I} :
\begin{align}\label{valu:1}
& v(f_{\rho}(y)(1+y)+(-1)^p)=v(y), \; \;  v(\tilde{s})=v(\rho_jy^{p^j}) , \; \; v(\sigma)=qv(\tilde{s}), \\
&v\bigg(\sum_{k=1}^{n-1}(\rho_ky)^{q/p^k}(-p)^{d_k}(-1)^qf_{\rho}(y)^{q(p^k-1)/p^k}\bigg)=v((\rho_ly)^{p^{n-l}}p^{d_l}). \nonumber
\end{align}
Then one checks that 
\begin{align}\label{valu:2}
 v(R_1)>v(a_ny) > v((\rho_ly)^{p^{n-l}}p^{d_l}) > v(\sigma).
\end{align}
It implies that $v((y^{q^2/p}-b_n)^p)=qv(\tilde{s})$, so one considers $(y^{q^2/p}-b_n+\tilde{s}^{q/p})^p$. By expanding this last expression, using \eqref{valu:1}, \eqref{valu:2} and taking into account
\begin{align*}
v(\sum_{k=1}^{q-1}\binom{q}{k}\tilde{s}^ky^{q(q-k)})>v(a_ny), \; \; v(\sum_{k=1}^p\binom{p}{k}(y^{q^2/p}-b_n)^k\tilde{s}^{(p-k)q/p})>v(a_ny),
\end{align*}
one obtains that $pv(y^{q^2/p}-b_n+\tilde{s}^{q/p})=v((\rho_ly)^{p^{n-l}}p^{d_l})$, leading us to consider 
\[(y^{q^2/p}-b_n+\tilde{s}^{q/p}+(\rho_ly)^{q/p^{l+1}}(-p)^{d_l/p}f_{\rho}(y)^{q(p^l-1)/p^{l+1}})^p. \]
By expanding this expression and using \eqref{valu:1} and \eqref{valu:2} one easily checks that it has valuation $ v((\rho_{l_1}y)^{p^{n-{l_1}}}p^{d_{l_1}})$ where $l_1 := \min \lbrace l+1  \leq  i \leq n-1, \rho_i \neq 0 \rbrace$. By induction one shows that 
\begin{equation}\label{exp:t}
 t:=y^{q^2/p}-b_n+\tilde{s}^{q/p}+\sum_{k=1}^{n-1}(\rho_ky)^{q/p^{k+1}}(-p)^{d_k/p}f_{\rho}(y)^{q(p^k-1)/p^{k+1}},
\end{equation}
satisfies $pv(t)=v(a_ny)$. Then $v_{\L}(p^{q^2}t^{-(p-1)(q+1)})=v_{\L}(p)/q^2= [\L:\Qpur]/q^2$, so $q^2$ divides $[\L:K]$. It implies that $L_{\rho}(X)$ is irreducible over $K$.\\

\noindent \textbf{Step B :} \textsl{Reduction step.}

\noindent The last non-trivial group $G_{i_0}$ of the lower ramification filtration $(G_i)_{i \geq 0}$ of $G:=\Gal(\M/K)$  is a subgroup of ${\rm Z}(G)$ (\cite{Ser} IV \S 2 Corollary 2 of Proposition 9) and as ${\rm Z}(G) \simeq \Z/p\Z$, it follows that $G_{i_0} ={\rm Z}(G)$.\\
\indent According to \textbf{Step VIII} the group $H:=\Gal(\M/ \L)$ is ${\rm Z}(G)$. Consequently, the filtration $(G_i)_{i \geq 0}$ can be deduced from that of $\M/\L$ and $\L/K$ (see \cite{Ser} IV \S 2 Proposition 2 and Corollary of Proposition 3).\\

\noindent \textbf{Step C :} \textsl{Let $\sigma \in \Gal(\L/K) - \lbrace 1 \rbrace$, then $v(\sigma(t)-t)=q^2v(\pi_K)$.}

\noindent Let $y':=\sigma(y)$, one deduces the following easy lemma from \textbf{Step VII}.
\begin{lem}\label{lem:stepC}
For any $n \geq 0$, $v(y^n-y'^n) \geq nv(y)$.
\end{lem}
\noindent Recall the definition $\tilde{s}:=2\rho_0y+\sum_{k=1}^{n-1}\rho_ky^{p^k}$. First one shows that modulo  ${(y-y')^{q^2/p}\mathfrak{m}}$ one has 
\begin{align}\label{equa:8}
\sigma(\tilde{s})^{q/p}-\tilde{s}^{q/p}  \equiv (2\rho_0)^{q/p}(y'^{q/p}-y^{q/p})+\sum_{k=1}^{n-1}\rho_k^{q/p}(y'^{qp^k/p}-y^{qp^k/p}).
\end{align}
Indeed, let $(m_i)_{i=0,\dots,n-1} \in \N^n$ such that $m_0+m_1+\dots+m_{n-1}=q/p$ and $t:=m_0+m_1p+\dots+m_{n-1}p^{n-1}$, then using lemma \ref{lem:stepC} one checks that
\begin{align*}
v(p\rho_0^{m_0}\rho_1^{m_1}\dots\rho_{n-1}^{m_{n-1}}(y^t-y'^t))>\frac{q^2}{p}v(y-y').
\end{align*}
This inequality implies  \eqref{equa:8}.\\
\indent Let $1 \leq k \leq n-1$ and write $f_{\rho}(y)^{(p^k-1)q/p^{k+1}}=1+\sum_{i \in I_k}\alpha_{i,k}y^i$, for some set $I_k$. Then 
\begin{align*}
&y'^{q/p^{k+1}}f_{\rho}(y')^{(p^k-1)q/p^{k+1}}-y^{q/p^{k+1}}f_{\rho}(y)^{(p^k-1)q/p^{k+1}}\\
 = & \; y'^{q/p^{k+1}}-y^{q/p^{k+1}}+\sum_{i \in I_k}\alpha_{i,k}(y'^i-y^i).
\end{align*} 
Let $i \in I_k$. Consider the case when $v(\alpha_{i,k}) \geq v(\rho_h)$ for some $0 \leq h \leq n-1$, then using \textbf{Step VII}, one checks that $\forall \; 1 \leq k \leq n-1$, $v(\alpha_{i,k}) \geq v(\rho_h)>qv(y'-y)/p^{k+1}$. If this case does not occur, then according to the expression of $f_{\rho}(y)$ one has $i \geq q/p^{k+1}+q$ and using lemma \ref{lem:stepC} one checks that $v(y'^i-y^i)> qv(y'-y)/p^{k+1}$. In any case $v(\alpha_{i,k}(y'^i-y^i)) >  qv(y'-y)/p^{k+1}$ and one checks that
\begin{align}\label{equa:9}
v(p^{d_k/p}\rho_k^{q/p^{k+1}}\alpha_{i,k}(y'^i-y^i))> q^2v(y'-y)/p. 
\end{align}
Taking into account \eqref{exp:t}, \eqref{equa:8} and \eqref{equa:9}, one gets $\md \; (y'-y)^{q^2/p}\mathfrak{m}$ 
\begin{align}\label{equa:10}
\sigma(t)-t & \equiv \;y'^{q^2/p}- y^{q^2/p}+(2\rho_0)^{q/p}(y'^{q/p}-y^{q/p})\\
+&\sum_{k=1}^{n-1}\rho_k^{q/p}(y'^{qp^k/p}-y^{qp^k/p})+\sum_{k=1}^{n-1}(-p)^{d_k/p}\rho_k^{q/p^{k+1}}(y'^{q/p^{k+1}}-y^{q/p^{k+1}}).\nonumber
\end{align}
Using lemma \ref{lem:stepC}, it is now straight forward to check the following relations $\md \; (y'-y)^{q^2/p}\mathfrak{m}$  .
\begin{align*}
y'^{q^2/p}-y^{q^2/p} &\equiv (y'-y)^{q^2/p},\\
\rho_k^{q/p}(y'^{qp^k/p}-y^{qp^k/p}) &\equiv \rho_k^{q/p}(y'-y)^{qp^k/p},\\
(-p)^{d_k/p}\rho_k^{q/p^{k+1}}(y'^{q/p^{k+1}}-y^{q/p^{k+1}}) &\equiv (-p)^{d_k/p}\rho_k^{q/p^{k+1}}(y'-y)^{q/p^{k+1}}.
\end{align*}
Using \textbf{Step VII}, one sees that each of these three elements has valuation $q^2v(y'-y)/p$, thus one gets
\begin{align}\label{equa:11}
(\sigma(t)-t)^p \equiv &(y'-y)^{q^2}+(2\rho_0)^q(y'-y)^q+\sum_{k=1}^{n-1}\rho_k^q(y'-y)^{qp^k}\\
+&\sum_{k=1}^{n-1}(-p)^{d_k}\rho_k^{q/p^{k}}(y'-y)^{q/p^{k}}  \;\md \;   (y'-y)^{q^2}\mathfrak{m}.\nonumber
\end{align}
Now  recall  \textbf{Step VII}, the definitions of the $\rho_k$'s and of $\lambda$, then for some $v \in R^{\times}$ and $\Sigma \in R$
\begin{align*}
\rho_k=&u_k\lambda^{p(q-p^k)/(1+q)},\; \; y'-y=v\lambda^{p/(1+q)}\; \; {\rm and} \;  -p=\lambda^{p-1}+p\lambda\Sigma.
\end{align*}
Since $q^2v(y'-y)=\frac{pq^2}{1+q}v(\lambda)$, equation \eqref{equa:11} becomes
\begin{align*}
(\sigma(t)-t)^p \equiv \lambda^{\frac{q^2p}{1+q}}\big[ v^{q^2}+(2u_0v)^q+\sum_{k=1}^{n-1}(u_k^qv^{qp^k}+(u_kv)^{q/p^k})\big] \md \lambda^{\frac{q^2p}{1+q}}\mathfrak{m}.
\end{align*}
From the action of $\sigma$ on the stable reduction (see \textbf{Step VIII}), one has that the automorphism of $\Proj_k$ given by $t  \mapsto t+ \bar{v}$ has a prolongation to $A_u/k$, so Proposition \ref{lem:prolon} implies that 
\begin{align}\label{equa:12} \bar{v}^{q^2}+(2\bar{u}_0\bar{v})^q+\sum_{k=1}^{n-1}(\bar{u}_k^q\bar{v}^{qp^k}+(\bar{u}_k\bar{v})^{q/p^k}) + \bar{v}=0.
\end{align}
Assume that $\bar{v}^{q^2}+(2\bar{u}_0\bar{v})^q+\sum_{k=1}^{n-1}(\bar{u}_k^q\bar{v}^{qp^k}+(\bar{u}_k\bar{v})^{q/p^k}) =0$,
then from \eqref{equa:12} one has $\bar{v}=0$, which contradicts $v \in R^{\times}$. It implies that $v(\sigma(t)-t)=q^2v(\lambda)/(1+q)=q^2v(y-y')/p=q^2v(\pi_K)$.\\

\noindent \textbf{Step D :} \textsl{The ramification filtration of $L/K$ is :}
\begin{equation*}
(G/H)_0=(G/H)_1 \supsetneq (G/H)_2 = \lbrace 1 \rbrace.
\end{equation*}

\noindent Since $K/\Qpur$ is tamely ramified of degree $(p-1)(q+1)$, one has $K=\Qpur(\pi_K)$ with $\pi_K^{(p-1)(q+1)}=p$ for some uniformizer $\pi_K$ of $K$. In particular $z:=\pi_K^{q^2}/t$, is a uniformizer of $\L$. Let ${\sigma \in \Gal(\L/K)-\lbrace 1 \rbrace}$, then 
\begin{align*}
\sigma(z)-z=\frac{t-\sigma(t)}{\sigma(t)t}\pi_K^{q^2}=\frac{t-\sigma(t)}{\pi_K^{q^2}}\frac{\pi_K^{q^2}}{t}\frac{\pi_K^{q^2}}{\sigma(t)} .
\end{align*}
Using \textbf{Step C} one obtains $v(\sigma(z)-z)=2v(z)$, i.e. $(G/H)_2=\lbrace 1 \rbrace$.\\

\noindent \textbf{Step E :} \textit{From now on, we assume $p=2$. Let $s:=(q+1)(2q^2-1)$. There exist $u, h \in L$ and $r\in \pi_L^s\mathfrak{m}$ such that $v_L(2y^{q/2}h)=s$ and}
\begin{align*}
f_{\rho}(y)u^2=1+\rho_{n-1}y^{1+q/2}+2y^{q/2}h+r.
\end{align*}

\noindent To prove the first statement we note that, from the definition of $f_{\rho}(y)$, one has $f_{\rho}(y)=1+T$ with $v(T)=qv(y)$ and $L_{\rho}(y)=0$, thus
\begin{align*}
&\Big(\frac{s_{\rho}^{q/2}(y)}{b_n}\Big)^2=f_{\rho}(y)^{q-1}(1+y)+\sum_{k=1}^{n-1}\frac{(\rho_ky)^{q/2^k}}{2^{2+\dots+2^{n-k}}}f_{\rho}(y)^{q(2^k-1)/2^k},\\
{\rm and \;} &  f_{\rho}(y)^{q-1}(1+y)=1+y+\sum_{k=1}^{q-1}\binom{q-1}{k}T^k(1+y).
\end{align*}
Then, we put $\tilde{\Sigma}:=\sum_{k=1}^{q-1}\binom{q-1}{k}T^k(1+y)$ and 
\begin{align*}
h:=\frac{s_{\rho}^{q/2}(y)}{b_n}+\sum_{k=1}^{n-1}\frac{(\rho_ky)^{q/2^{k+1}}}{2^{1+\dots+2^{n-k-1}}}f_{\rho}(y)^{q(2^k-1)/2^{k+1}} -1.
\end{align*}
Then one computes 
\begin{align*}
h^2&=\Big[ \frac{s_{\rho}^{q/2}(y)}{b_n}+\sum_{k=1}^{n-1}\frac{(\rho_ky)^{q/2^{k+1}}}{2^{1+\dots+2^{n-k-1}}}f_{\rho}(y)^{q(2^k-1)/2^{k+1}}\Big]^2 +1 -2(h+1)\\
&=(\frac{s_{\rho}^{q/2}(y)}{b_n})^2+\sum_{k=1}^{n-1}\frac{(\rho_ky)^{q/2^{k}}}{2^{2+\dots+2^{n-k}}}f_{\rho}(y)^{q(2^k-1)/2^{k}}+\Sigma_1+1-2(h+1)\\
&=2+y+2\sum_{k=1}^{n-1}\frac{(\rho_ky)^{q/2^{k}}}{2^{2+\dots+2^{n-k}}}f_{\rho}(y)^{q(2^k-1)/2^{k}}+\Sigma_1+\tilde{\Sigma}-2(h+1).
\end{align*}
In \textbf{Step III}, we proved that $v(B_n)=qv(y)=2v(b_n)/q$ where $B_n=-s_{\rho}(y)$, so $v( \frac{s_{\rho}^{q/2}(y)}{b_n})=0$ and one checks using \textbf{Step I} that 
\begin{equation}\label{eqStepE}
v(2) > v(y),{\; \rm and \;}\forall \; 1 \leq k \leq n-1, \; v\Big(\frac{(\rho_ky)^{q/2^{k+1}}}{2^{1+\dots+2^{n-k-1}}}\Big) \geq 0,
\end{equation}
thus $v(h+1) \geq 0$ and $v(2(h+1))\geq v(2) >v(y)$.  One checks in the same way that $v(\Sigma_1)>v(y)$. One has $v(\tilde{\Sigma}) \geq v(T) > v(y)$, so $v(h^2)=v(y)$ and $v_L(2y^{q/2}h)=s$.\\
\indent To prove the second statement of \textbf{Step E}, we first remark that $\forall i \geq 1$ $f_{\rho}(y)^i=1+\sum_{k=1}^i\binom{i}{k}T^k=1+\Sigma_i$, whence $v(\Sigma_i) \geq v(T)$. Since, for all $ 0 \leq k \leq n-1$, $v(\rho_ky^{p^k})>qv(y)$ one has $\md \pi_L^s\mathfrak{m}$
\begin{align}\label{equa:13}
\frac{s_{\rho}^{q/2}(y)}{b_n}2y^{q/2} \equiv \big[(2\rho_0y)^{q/2}&+\sum_{k=1}^{n-1}(\rho_ky^{2^k})^{q/2}+y^{q^2/2}\big]\frac{y^{q/2}}{2^{2+\dots+2^{n-1}}}.
\end{align}
One also checks that $\forall i \geq 1$ , $v_L(2y^{q/2}\Sigma_i)>s$, then according to \eqref{eqStepE}, $\forall \; i \geq 1$ and $1 \leq k \leq n-1$
\begin{align}\label{equa:14}
 &v_L\Big(\frac{(\rho_ky)^{q/2^{k+1}}}{2^{1+\dots+2^{n-k-1}}}2y^{q/2}\Sigma_i\Big) > s  {\rm \; and \; one \; checks\; that \; } v_L\big(\frac{(2\rho_0)^{q/2}y^q}{2^{2+\dots+2^{n-1}}}\big) > s.
\end{align}
Thus, applying relations \eqref{equa:13}, \eqref{equa:14} and the definition of $h$, one has 
\begin{align}\label{eqsimpl}
2hy^{q/2} \equiv & \big[\sum_{k=1}^{n-1}(\rho_ky^{2^k})^{q/2}+y^{q^2/2}\big]\frac{y^{q/2}}{2^{2+\dots+2^{n-1}}} \nonumber\\
&+\sum_{k=1}^{n-1}\frac{(\rho_ky)^{q/2^{k+1}}}{2^{1+\dots+2^{n-k-1}}}2y^{q/2}-2y^{q/2} \md  \pi_L^s\mathfrak{m}.
\end{align}
Finally one puts
\begin{align*}
u:=1-y^{q/2}-\sum_{k=0}^{n-2}\frac{y^{2^k(1+q)}}{2^{1+\dots+2^k}}+\sum_{i=1}^{n-1}\sum_{k=n-i-1}^{n-2}\frac{\rho_i^{2^k}}{2^{1+\dots+2^k}}y^{2^k(1+2^i)}=1+\tilde{u},
\end{align*}
and one checks that $v(\tilde{u})=v(y^{q/2})$. From the equality
\begin{align*}
f_{\rho}(y)u^2-1=\sum_{k=0}^{n-1}\rho_ky^{1+2^k}+y^q+y^{1+q}+(1+T)2\tilde{u}+(1+T)\tilde{u}^2,
\end{align*}
taking into account that $v_L(2T\tilde{u})>s$, $v_L(T\tilde{u}^2)>s$, $\forall \; 0 \leq k \leq n-2$, $ v_L(\rho_ky^{1+2^k}) > s$ and expanding $\tilde{u}$ and $\tilde{u}^2$ one gets modulo $\pi_L^s \mathfrak{m}$ 
\begin{align}\label{equa:15}
f_{\rho}(y)u^2-1 \equiv &\rho_{n-1}y^{1+q/2}-2y^{q/2}+2y^q -\sum_{k=1}^{n-2}\frac{2y^{2^k(1+q)}}{2^{1+\dots+2^k}}+\sum_{k=1}^{n-1}\frac{y^{2^k(1+q)}}{2^{2+\dots+2^k}}+ \nonumber\\
&+\sum_{i=1}^{n-1}\sum_{k=n-i-1}^{n-2}\frac{2\rho_i^{2^k}y^{2^k(1+2^i)}}{2^{1+\dots+2^k}}+\sum_{i=1}^{n-1}\sum_{k=n-i}^{n-1}\frac{\rho_i^{2^k}y^{2^k(1+2^i)}}{2^{2+\dots+2^k}}.
\end{align}
Arranging the terms  of \eqref{equa:15} , taking into account that ${v_L(2y^q)>s}$ and for all  $2 \leq i \leq n-1$ and $n-i \leq k \leq n-2$
\begin{equation*}
v_L\Big(\rho_i^{2^k}y^{2^k(1+2^i)}\frac{2}{2^{2+\dots+2^k}}\Big) > s,
\end{equation*}
and comparing with \eqref{eqsimpl}, one obtains $f_{\rho}(y)u^2-1 \equiv \rho_{n-1}y^{1+q/2}+2hy^{q/2}\;  \md \pi_L^s\mathfrak{m}$.\\

\noindent \textbf{Step F:} \textit{The ramification filtration of $M/L$ is}
\begin{equation*}
H_0=H_1 = \dots = H_{1+q} \supsetneq \lbrace 1 \rbrace.
\end{equation*}
\noindent One has to show that $v_{M}(\mathcal{D}_{M/L})=q+2$, we will use freely results from \cite{Ser} IV. If $\rho_{n-1} = 0$, then according to \textbf{Step E}, one has
\begin{align*}
f_{\rho}(y)u^2=1+2y^{q/2}h+r,
\end{align*}
and one concludes using \cite{CM} Lemma 2.1. Else, if $\rho_{n-1} \neq 0$, one has
\begin{align*}
\max_{u\in L^{\times}}v_{L}(f_{\rho}(y)u^ 2-1) \geq v_L(\rho_{n-1}y^{1+q/2}),
\end{align*}
then  \cite{LRS} Lemma 6.3 implies that ${v_{M}(\mathcal{D}_{M/L}) \leq q+3}$. Using \textbf{Step B}, \textbf{Step D} and \cite{Ser} IV \S 2 Proposition 11, one has that the break in the ramification filtration of $M/L$ is congruent to $1 \md 2$, i.e. ${v_{M}(\mathcal{D}_{M/L}) \leq q+2}$. According to \textbf{Step D} and lemma \ref{maurizio}, the break $t$ of $M/L$  is in $1+q\N$. If $t=1$ then $G_2= \lbrace 1 \rbrace$ and $G_1/G_2 = G/G_2 \simeq G$  would be abelian, so $t \geq 1+q$, i.e. $v_M(\mathcal{D}_{M/L}) \geq q+2$.\\

\noindent \textbf{Step G :} \textsl{Computations of conductors.}

\noindent For $l \neq 2$ a prime number, the $G$-modules $\Jac(C)[l]$ and $\Jac(\mathcal{C}_k)[l]$ being isomorphic one has  that for $i \geq 0$ :
\[ \dim_{\F_l} \Jac(C)[l]^{G_i}= \dim_{\F_l}\Jac(\mathcal{C}_k)[l]^{G_i}.
\]
Moreover, for $0 \leq i \leq 1+q$ one has $\Jac(\mathcal{C}_k)[l]^{G_i} \subseteq \Jac(\mathcal{C}_k)[l]^{{\rm Z}(G)}$ , then from $\mathcal{C}_k/{\rm Z}(G) \simeq \Proj_k$ and lemma \ref{guralnick} it follows that  for $0 \leq i \leq 1+q$, $\dim_{\F_l}\Jac(\mathcal{C}_k)[l]^{G_i} = 0$. Since $g(C)=q/2$ one gets $f(\Jac(C)/K)=2q+1$ and $\sw(\Jac(C)/\Q_2^{\rm ur})=1$.
\end{proof}

\textbf{Example :} Magma codes are available on the author webpage. Let ${K:=\Q_2^{{\rm ur}}(2^{1/5})}$ and $f(X):=1+2^{6/5}X^2+2^{4/5}X^3+X^4+X^5 \in K[X]$, one checks that the smooth, projective, integral curve birationally given by ${Y^2=f(X)}$ has the announced properties, that is the wild monodromy $M/K$  has degree $32$ and one can describe its ramification filtration. The first program checks that \textbf{Step A} and \textbf{Step D} hold for this example. The second program checks \textbf{Step F} and is due to Guardia, J., Montes, J. and Nart, E. (see \cite{GMN}) and computes $v_M(\mathcal{D}_{M/\Q_2^{{\rm ur}}})=194$. Using \cite{Ser} III \S4  Proposition 8, one finds that $v_M(\mathcal{D}_{M/K})=66$, which was the announced result in Theorem \ref{maintheo} \textit{3}.\\

\noindent \textbf{Remarks : }
\begin{enumerate}
\item The above example was the main motivation for \textbf{Step F} since it shows that one could expect the correct behaviour for the ramification filtration of $\Gal(M/K)$ when $p=2$.
\item The naive method to compute the ramification filtration of $M/K$ in the above example fails. Indeed, in this case Magma needs a huge precision when dealing with $2$-adic expansions to get the correct discriminant.
\end{enumerate}

\noindent \textbf{Acknowledgements :} I would like to thank M. Monge for pointing out lemma \ref{maurizio}. 
\bibliographystyle{alpha} 
\bibliography{liftutf8v2}

\begin{thebibliography}{GMN11}

\bibitem[BK05]{Br-Kr}
A.~Brumer and K.~Kramer.
\newblock {\em The conductor of an abelian variety}.
\newblock {\em {\rm Compositio mathematica}}, n 2({\bf 92}), 2005.

\bibitem[CM11]{CM}
P.~Chrétien and M.~Matignon.
\newblock {\em Maximal monodromy in unequal characteristic}.
\newblock {\em {\rm submitted}}, 2011.

\bibitem[GMN11]{GMN}
J.~Guardia, J.~Montes, and E.~Nart.
\newblock {\em Higher Newton polygons in the computation of discriminants and
  prime ideal decomposition in number fields}.
\newblock {\em {\rm Journal de théorie des nombres de Bordeaux}}, 2, 2011.

\bibitem[Gur03]{Gur}
R.~Guralnick.
\newblock {\em Monodromy groups of coverings of curves}.
\newblock In {\em {\rm Galois groups and fundamental groups}}, volume~{\bf 41}.
  {\rm MSRI Publications}, 2003.

\bibitem[Hup67]{Hu}
B.~Huppert.
\newblock {\em Endliche Gruppen I}.
\newblock {\em {\rm Grundlehren der Mathematischen Wissenschaften}}, ({\bf
  134}), 1967.

\bibitem[Kra90]{Kr}
A.~Kraus.
\newblock {\em Sur le défaut de semi-stabilité des courbes elliptiques à
  réduction additive}.
\newblock {\em {\rm Manuscripta Mathematica}}, ({\bf 69}), 1990.

\bibitem[Liu02]{Liu}
Q.~Liu.
\newblock {\em {\em Algebraic Geometry and Arithmetic Curves}}.
\newblock Oxford University Press, 2002.

\bibitem[LM05]{LM}
C.~Lehr and M.~Matignon.
\newblock {\em Automorphism groups for $p$-cyclic covers of the affine line}.
\newblock {\em {\rm Compositio Mathematica}}, n 5({\bf 141}), 2005.

\bibitem[LRS93]{LRS}
P.~Lockhart, M.I. Rosen, and J.~Silverman.
\newblock {\em An upper bound for the conductor of an abelian variety}.
\newblock {\em {\rm Journal of algebraic geometry}}, n 2, 1993.

\bibitem[Ogg67]{Ogg}
A.P. Ogg.
\newblock {\em Elliptic curves and wild ramification}.
\newblock {\em {\rm American Journal of Mathematics}}, 89({\bf 1}), 1967.

\bibitem[Ray90]{Ra}
M.~Raynaud.
\newblock {\em $p$-groupes et réduction semi-stable des courbes }.
\newblock In {\rm Birkhäuser}, editor, {\em {\rm The Grothendieck Festschrift,
  Vol. III}}, 1990.

\bibitem[Ser67]{Ser2}
J.-P. Serre.
\newblock {\em {\em Représentations linéaires des groupes finis}}.
\newblock Hermann, Paris, 1967.

\bibitem[Ser79]{Ser}
J.-P. Serre.
\newblock {\em {\em Local Fields}}.
\newblock Graudate Texts in Mathematics ({\bf 67}), 1979.

\bibitem[ST68]{SerTat}
J.-P. Serre and J.~Tate.
\newblock {\em Good reduction of abelian varieties}.
\newblock {\em {\rm Annals of Mathematics}}, ({\bf 88}), 1968.

\bibitem[Suz86]{Su}
M.~Suzuki.
\newblock {\em {\em Group Theory II}}.
\newblock Grundlehren der Mathematischen Wissenschaft ({\bf 248}), 1986.

\end{thebibliography}

\end{document}